\documentclass[12pt,twoside]{amsart}

\usepackage{fullpage}
\usepackage{amssymb,mathrsfs}
\usepackage{bbm}
\usepackage{amsfonts}
\usepackage{amsmath, amsthm, amssymb, mathrsfs, paralist,tabularx,supertabular, verbatim, amsthm}
\usepackage[all]{xy}

\newcommand{\Z}{{\mathbb Z}}
\newcommand{\Q}{{\mathbb Q}}

\newcommand{\F}{{\mathbb F}}
\newcommand{\gl}{\text{GL}_2}
\newcommand{\tr}{\text{tr }}

\include{python}

\evensidemargin \oddsidemargin
\parskip=\medskipamount

\newtheorem{thm}{Theorem}[section]
\newtheorem{theorem}[thm]{Theorem}
\newtheorem{corollary}[thm]{Corollary}
\newtheorem{lemma}[thm]{Lemma}	
\newtheorem{proposition}[thm]{Proposition}

\newtheorem{conjecture}[thm]{Conjecture}

\theoremstyle{definition}
\newtheorem{definition}[thm]{Definition}

\theoremstyle{remark}

\newtheorem*{lemma*}{Lemma}

\numberwithin{equation}{section}

\title{Distribution of squarefree values of sequences associated with elliptic curves}

\author{Shabnam Akhtari}
\author{Chantal David}
\author{Heekyoung Hahn}
\author{Lola Thompson}

\address{University of Oregon\\
Department of Mathematics\\
Fenton Hall\\
Eugene, OR 97403\\
United States}
\email[] {akhtari@uoregon.edu}

\address{Concordia University\\
Department of Mathematics and Statistics\\
1455 de Maisonneuve West\\
Montr\'eal, Qu\'ebec\\
Canada H3G 1M8}
\email[] {cdavid@mathstat.concordia.ca}

\address{Duke University\\
Department of Mathematics\\
Box 90320\\
Durham, NC 27708\\
United States}
\email[] {hahn@math.duke.edu}

\address{University of Georgia\\
Department of Mathematics\\
Boyd Graduate Studies Research Center\\
Athens, GA 30602\\
United States}
\email[] {lola@math.uga.edu}

\begin{document}

\begin{abstract}

Let $E$ be a non-CM elliptic curve defined over $\mathbb{Q}$. For each prime $p$ of good reduction, $E$ reduces to a curve $E_p$ over the finite field $\mathbb{F}_p$. For a given squarefree polynomial $f(x,y)$, we examine the sequences $f_p(E) := f(a_p(E), p)$, whose values are associated with the reduction of $E$ over $\F_p$. We are particularly interested in two sequences: $f_p(E) =p + 1 - a_p(E)$ and $f_p(E) = a_p(E)^2 - 4p$. We present two results towards the goal of determining how often the values in a given sequence are squarefree. First, for any fixed curve $E$, we give an upper bound for the number of primes $p$ up to $X$ for which $f_p(E)$ is squarefree. Moreover, we show that the
conjectural asymptotic for the prime counting function $$\pi_{E,f}^{SF}(X) := \#\{p \leq X: f_p(E) \ \hbox{is squarefree}\}$$ is consistent with the asymptotic for the average over curves $E$ in a suitable box.

\end{abstract}

\maketitle

\section{Introduction}

Let $E$ be an elliptic curve over $\Q$. For each prime $p$ of good reduction, $E$ reduces to a curve $E_p$ over the finite field $\F_p$ with $|E_p(\F_p)| = p + 1 - a_p(E)$ and $|a_p(E)| \leq 2 \sqrt{p}$
(the Hasse bound). There are many open conjectures about the distribution of invariants associated with the
reductions of a fixed elliptic curve over $\Q$ to curves over the finite fields $\F_p$ as $p$ runs through the primes; the conjecture
of Lang and Trotter~\cite{lt} and the conjecture of
Koblitz~\cite{koblitz} are two well-known examples.
The Koblitz Conjecture concerns the number of primes $p\le X$ such that $|E(\F_p)|$ is prime, and
is thus analogous to the twin prime conjecture in the context of elliptic curves.
The fixed trace Lang-Trotter Conjecture concerns the number of primes $p\le X$ such that
the trace of Frobenius $a_p(E)$ is equal to a fixed integer $t$.  Another
conjecture of Lang and Trotter (also called the Lang-Trotter Conjecture)
concerns the number of primes
$p\le X$ such that the Frobenius field $\Q(\sqrt{a_p(E)^2-4p})$ is a fixed
imaginary quadratic field $K$. These conjectures are still completely open. In particular,
the only known lower bound for any of the conjectures described above is a result of Elkies \cite{elkies},
who proved that there are infinitely many supersingular primes (or equivalently, infinitely many
primes such that $a_p(E)=0$).

In this paper, we consider the question of counting the squarefree values in a sequence associated to the reductions $E_p$
 over the finite fields $\F_p$ of a fixed elliptic curve $E$ defined over $\Q$.
Two sequences are of particular interest (and were studied in previous work), namely
$|E_p(\F_p)| = p + 1 - a_p(E)$ and
$a_p(E)^2 - 4p$. The latter sequence is of interest since $\Z [ \sqrt{a_p(E)^2 - 4p} ]$ is the ring generated by the Frobenius element over $\F_p$; thus, it is related to the second conjecture of Lang and Trotter discussed above.

In general,
let $f(x,y) \in \Z[x,y]$ be squarefree. We consider the general sequence
$$\left\{ f_p(E) := f(a_p(E), p) \;:\; p \;\mbox{prime} \right\}$$
associated to a given elliptic curve $E$ over $\Q$.

We define
$$\pi_{E, f}^{SF}(X) := \# \{p \leq X : f_p(E) \ \hbox{is squarefree}\}.$$

It is not difficult to predict the precise asymptotic that one should obtain for
$\pi_{E, f}^{SF}(X)$ but the precise order of $\pi_{E, f}^{SF}(X)$ is not
known unconditionally for any sequence $f_p(E)$. If $E$ is a non-CM elliptic curve defined over $\mathbb{Q}$, then
 assuming the Generalized Riemann Hypothesis, the Pair Correlation Conjecture, and Artin Holomorphy Conjecture, Cojocaru showed in her thesis \cite{cojocaru-thesis} how to obtain the correct asymptotic
for $\pi_{E, f}^{SF}(X)$ when $f_p(E) = p + 1 - a_p(E)$. Her proof presumably extends to
other sequences.
For elliptic curves with complex multiplication, Cojocaru \cite{cojocaru} obtained the correct proportion of primes $p$ for which the sequence $p+1-a_p(E)$ is squarefree. Her asymptotic estimate relies heavily on the algebraic properties that CM elliptic curves possess; the same methods do not appear to be capable of handling the non-CM case. For
CM curves,
 handling the sequence $a_p(E)^2 - 4p$ requires a different approach, as computing the proportion of primes for which $a_p(E)^2 - 4p$ is squarefree is equivalent to counting the number of primes in a given quadratic progression. For example, let $E$ be the CM elliptic
curve $y^2 = x^3 - x$ with complex multiplication by the ring of Gaussian integers $\Z[i]$. Let $p$ be an ordinary prime that is congruent to $1$ modulo $4$. Since $E$ has rational 2-torsion, then $a_p(E)$ is even and $4$
divides $a_p(E)^2 - 4p$. We want to know when $(a_p(E)^2 - 4p)/4$ is squarefree.
Since $E$ has complex multiplication by $\Z[i]$, if $a_p(E) \neq 0$, then $a_p(E)^2 - 4p = -4 \alpha^2$
for some $\alpha \in \Z$,  and $(a_p(E)^2 - 4p)/4$ is squarefree
if and only if $\alpha = 1$ if and only if $p = (a_p(E)/2)^2 + 1.$
This latter problem remains a well-known open question.

To gain evidence for conjectures related to the distribution of invariants associated with the
reductions of a fixed elliptic curve over the finite fields $\F_p$, it is natural to consider the averages for these conjectures over
some family of elliptic curves.
This has been done by various authors originating with the work of Fouvry and
Murty \cite{FM} for the number of supersingular primes (i.e., the fixed trace Lang-Trotter
Conjecture for $t=0$).  See \cite{DP99}, \cite{DP04}, \cite{Jam04}, \cite{BBIJ05}, \cite{JS11}, and \cite{CFJ+11}
for other averages regarding the fixed trace Lang-Trotter Conjecture.
The average order for the Koblitz Conjecture was considered in \cite{BaCoDa}.
Very recently, the average has been successfully carried out for the Lang-Trotter Conjecture
on Frobenius fields \cite{CoIwJo}. In \cite{du}, the authors considered the average
of $\pi_{E, f}^{SF}(X)$ for $f_p(E) = a_p(E)^2 - 4p$ and showed that the conjecture holds on average
when the size of the family is large enough. This is equivalent to determining the
average over the finite fields $\F_p$, namely
$
\sum_{p \leq X} \# \left\{ E / \F_p \;:\; a_p(E)^2 - 4p \quad \text{is squarefree} \right\}.
$
For the sequence $f_p(E) = p+1-a_p(E)$,
the number of squarefree values was also
investigated over the finite fields $\F_p$ for $p \leq X$ by Gekeler \cite{gekeler}.
As a corollary to his result, one can show that the number of primes $p \leq X$  such that
$p+1-a_p(E)$ is squarefree follows the predicted asymptotic on average over all elliptic curves.

All of the aforementioned averages provide evidence for the stated conjectures, as they demonstrate that the average
asymptotic is on the same order of magnitude as the conjectured asymptotic for any given elliptic curve. In each case, the average asymptotic involves
a constant, which depends on the precise conjecture that is averaged, but does not necessarily correspond
to the constant that appears in the conjecture for every elliptic curve. It is therefore interesting to investigate
whether the average results are compatible with the corresponding conjectures at the level of the constants,
i.e., whether the average of the conjectured constants is equivalent to the constant obtained via the average conjecture.
This was done by Jones \cite{jones} for both the Lang-Trotter conjecture and the Koblitz conjecture. In this paper, we show that the same principle holds for the constants associated with the number of squarefree values of $f_p(E)$. Precise statements of our results are given in the next section.

\section{Statement of results}

It is not difficult to obtain an upper bound of the correct order of magnitude for
$\pi_{E, f}^{SF}(X)$ using the M\"{o}bius function
to detect squares, along with an explicit version of the Chebotarev Density theorem to count
$\# \left\{ p \leq X \;:\; d^2 \mid f_p(E) \right\}.$ Furthermore, one gets
the correct order of magnitude with the correct conjectural constant. In order to
give an expression for this constant, we need some definitions.
Let
$f(x,y) \in \Z[x,y]$ be squarefree.
Let
\begin{eqnarray} \label{def-cf}
C_{f}(n) &=& \{ g \in   \mbox{GL}_2(\Z / n \Z) : f(\tr{g},\det{g}) \equiv 0 \mod{n}\}.
\end{eqnarray}
For any elliptic curve $E$ over $\Q$, and any positive
integer $n$,  let $G_E(n)$ be the subgroup of
$\mbox{GL}_2(\Z / n \Z)$ defined
in Section \ref{torsionfields},
and let $M_E$ be the integer defined in Section \ref{serre-curves}.
We then define
\begin{eqnarray} \label{def-cef}
C_{E, f}(n) &=& \{ g \in G_E(n) : f(\tr{g}, \det{g}) \equiv 0 \mod{n}\}.
\end{eqnarray}
Then,
\begin{eqnarray}
\label{defCE}
C_{E,f}^{SF} = \prod_{\ell \nmid M_E}\left(1 - \frac{|C_{f}(\ell^2)|}{|\mathrm{GL}_2(\Z/\ell^2\Z)|}\right) \sum_{n \mid M_E} \mu(n) \frac{|C_{E,f}(n^2)|}{|G_E(n^2)|}. \end{eqnarray}

Our first result is the following:

\begin{theorem}\label{upperbound}  Let $E$ be a non-CM elliptic curve defined over $\mathbb{Q}$. For $X$ sufficiently large (depending on $E$), and any $\varepsilon> 0$,
we have $$\pi_{E, f}^{SF}(X) \leq C_{E,f}^{SF} \; \pi(X)\left(1 + O\left(\frac{1}{(\log \log X)^{1-\varepsilon}}\right)\right),$$ where $C_{E,f}^{SF}$
is the constant given in \eqref{defCE}. \end{theorem}

Our theorem provides evidence for the
conjectural  number of squarefree values in sequences $f_p(E)$ associated with
elliptic curves.

\begin{conjecture} \label{conjecture-SF} Let $E$ be a non-CM elliptic curve defined over $\mathbb{Q}$. As $X \rightarrow \infty$, we have
$$\pi_{E,f}^{SF}(X) \sim C_{E,f}^{SF} \; \pi(X),$$
where
$C_{E,f}^{SF}$ is the constant given in \eqref{defCE}.
\end{conjecture}

As mentioned in the previous section, Conjecture \ref{conjecture-SF} has been proven on average over the family
of all elliptic curves for some specific sequences $f_p(E)$. Let $E(a, b)$ denote the elliptic curve given by the equation
$$
y^2=x^3+ax+b,
$$with $4a^3+27b^2\neq 0$. Let $A$ and $B$ be positive constants. We define
\begin{eqnarray} \label{familyEC}
\mathcal{C}(A, B) : = \{E(a, b) : |a| \leq A \text{ and } |b| \leq B\}.
\end{eqnarray}

The following average results are due to David and Urroz, and Gekeler, respectively.

\begin{theorem} \cite{du} \label{thmdu} Let $f(x,y) = x^2 - 4y$ such that $f_p(E) = a_p(E)^2 - 4p$. Then
for any $\varepsilon > 0$, and any $A,B$ such that
$AB > x \log^8{x}$ with $A,B > x^\varepsilon$, we have as $X \rightarrow \infty$
$$
\frac{1}{|\mathcal{C}(A, B)|} \sum_{E \in \mathcal{C}(A, B)} \pi_{E, f}^{SF}(X) \sim C_f^{SF} \pi(X)
$$
where
$$
C_f^{SF} = \prod_{\ell}
\left(1 - \frac{|C_{f}(\ell^2)|}{|\mathrm{GL}_2(\Z/\ell^2\Z)|}\right) = \frac{1}{3} \prod_{\ell \neq 2}  1 - \frac{\ell^2 + \ell - 1}{\ell^2 (\ell^2 -1)}.$$
\end{theorem}

\begin{theorem} \cite{gekeler} \label{thmge} If $f(x,y) = y+1-x$ such that $f_p(E) = p+1-a_p(E)$,
we have as $X \rightarrow \infty$
$$
\frac{\sum_{p \leq X} \# \left\{ E / \F_p \;:\; f_p(E) \quad \mbox{is squarefree} \right\}}
{\sum_{p \leq X}  \# \left\{ E / \F_p \right\}} \sim C_f^{SF}
$$
where
$$
C_f^{SF} = \prod_{\ell}
\left(1 - \frac{|C_{f}(\ell^2)|}{|\mathrm{GL}_2(\Z/\ell^2\Z)|}\right) =  \prod_{\ell}  1 - \frac{\ell^3 - \ell - 1}{\ell^2 (\ell^2 -1)(\ell - 1)}.
$$
\end{theorem}

The proofs of the average results stated in Theorems \ref{thmdu} and \ref{thmge}
are very different. For Theorem \ref{thmdu}, the authors use Deuring's Theorem
to count elliptic curves over $\F_p$ such that $a_p(E)^2 - 4p$ is squarefree, and the
theorem follows from taking an average of class numbers. For Theorem \ref{thmge}, the author uses
completely different techniques that rely on Howe's work on
counting points on the moduli spaces of elliptic curves over $\F_p$ with a given group structure.
In both cases, the average constant $C_f^{SF}$ follows from somewhat elaborate computations that are
particular to the sequence $f_p(E)$ being studied. For a general sequence
$f_p(E)$, one believes that we should have
$$
\frac{1}{|\mathcal{C}(A, B)|} \sum_{E \in \mathcal{C}(A, B)} \pi_{E, f}^{SF}(X) \sim C_f^{SF} \pi(X)
$$
where
$$C_f^{SF} := \prod_{\ell}\left(1 - \frac{|C_{f}(\ell^2)|}{|\mathrm{GL}_2(\Z/\ell^2\Z)|}\right).$$

We provide evidence for an average result of this nature by showing that the average of the conjectural constants $C_{E,f}^{SF}$ defined in \eqref{defCE}
coincide with the constant $C_f^{SF}$ for a general squarefree polynomial
$f \in \Z[x,y]$. This forms our second result.

\begin{theorem}\label{thmaveragec} Let $f \in \Z[x,y]$ be non-constant and squarefree, and
let $\mathcal{C}(A, B)$ be the family of curves defined in (\ref{familyEC}).
Then, we have
$$\frac{1}{|\mathcal{C}(A, B)|} \sum_{E \in \mathcal{C}(A, B)} C_{E, f}^{SF} \sim C_f^{SF}.$$
\end{theorem}

In particular, the constants appearing in Theorems \ref{thmdu} and \ref{thmge}
are indeed the average of the constants from Conjecture \ref{conjecture-SF}.

\begin{corollary}\label{averagec}
Let $f(x,y) = y+1-x$ or $x^2-4y.$ As $A,B \rightarrow \infty$, we have
$$\frac{1}{|\mathcal{C}(A, B)|} \sum_{E \in \mathcal{C}(A, B)} C_{E, f}^{SF} \sim C_f^{SF}.$$
\end{corollary}

%\begin{theorem}\label{piaverage}
%Fix any $\varepsilon >0$. Let $A$ and $B$ such that $A> X^{7/8+\varepsilon}$ and $B>??$. Then as $X\longrightarrow\infty$,
%$$
%\frac{1}{|\mathcal{C}(A, B)|} \sum_{E \in \mathcal{C}(A, B)} \pi_{E,f}^{SF}(X) = {C}_{E, f}^{SF} \pi(X)+O(?)
%$$
%\end{theorem}

%\begin{corollary}
%$\pi_{E,f}^{SF}(X) = {C}_{E,f}^{SF} \pi(X)$ except outside an exceptional set of size ****
%\end{corollary}

%{\bf Key point:} Having a shorter average improves the size of the exceptional set

We now outline the contents of this paper. In Section \ref{pre}, we set the notation and basic definitions,
and state some relevant results from the literature. The proof of Theorem \ref{upperbound} will be given in Section \ref{ub}. As in \cite{jones}, our proof of Theorem \ref{thmaveragec} requires computing separate averages over non-Serre curves and Serre curves. These computations are done in Sections \ref{nonserre} and \ref{serre}, respectively.

%-------------------------------------------------------
\section{Preliminaries}\label{pre}
%----------------------------------------------------------------

In this section, we introduce the notation and definitions which will be used throughout the paper. First,
we provide the necessary background on torsion fields attached to elliptic curves and their Galois groups,
as well as some information about Serre curves, which will be used in our proof of Theorem \ref{thmaveragec}.
We then state
an effective form of the Chebotarev Density Theorem, which will be used to prove Theorem \ref{upperbound}.

\subsection{Torsion fields of elliptic curves and Serre's theorem}\label{torsionfields}

For each positive integer $n$, let $E[n]$ be the group of $n$-torsion points of $E$. It is well-known that
$E[n] \simeq \Z / n \Z \times \Z / n \Z$ as an abstract abelian group.
Let $\mathbb{Q}\left( E[n]\right)$ denote the $n$th division field of $E$, obtained by adjoining to $\mathbb{Q}$ the $x$ and $y$-coordinates of the $n$-torsion points of $E$. This is a Galois extension of $\Q$, and
$\textrm{Gal}\left(\mathbb{Q}\left(E[n] \right) /\mathbb{Q} \right)$
acts on $E[n]$, giving rise to an injective group homomorphism
$$
\rho_{E, n} : \textrm{Gal}\left(\mathbb{Q}\left(E[n] \right) /\mathbb{Q} \right) \rightarrow \mathrm{GL}_{2}\left(\mathbb{Z}/ n\mathbb{Z}\right).
$$

\begin{definition} Let $G_E(n)$ denote the image of $\rho_{E, n}$ inside $\gl(\Z / n\Z).$ \end{definition}
Taking the inverse limit of the $\rho_{E, n}$ over positive integers $n$ (with a basis chosen compatibly), one obtains a continuous group homomorphism
$$
\rho_{E} : G_{\mathbb{Q}} \rightarrow \mathrm{GL}_{2}(\hat{\mathbb{Z}}),
$$
where $\hat{\mathbb{Z}} =\varprojlim \mathbb{Z}/ n\mathbb{Z}$, and $G_{\mathbb{Q}}
= \mbox{Gal}(\overline{\Q}/\Q).$

Serre proved the following theorem:

\begin{thm} \cite{Serre}
Suppose that $E$ is an elliptic curve over $\mathbb{Q}$ which has no complex multiplication. Then, with the notation defined as above, we have
$$
[\mathrm{GL}_{2}(\hat{\mathbb{Z}}) : \rho_{E}(G_{\mathbb{Q}})] < \infty.
$$
\end{thm}

Let $P(x)$ be a polynomial of degree $d$ with the leading coefficient $a$. The absolute logarithmic height of $P(x)$ is defined as
$$
h(P) = \frac{1}{d} \left( \log |a|+  \sum_{\alpha} \log\left( \max(1 , |\alpha|)\right) \right),
$$
where $\alpha$ ranges over all roots of polynomial $P(x)$.
The absolute logarithmic height of an algebraic number $\alpha$, denoted by $h(\alpha)$, is defined to be the absolute logarithmic height of its minimal polynomial. If $\alpha$ is a nonzero rational integer, then $h(\alpha) = \log|\alpha|$.

In this paper, we will need an effective version of Serre's theorem, which gives an explicit
bound on the index in terms of the parameters of the curve $E$. This is done in the following theorem, which is due to Zywina.

\begin{theorem} (\cite[Theorem 1.1]{zywina}) \label{explicitform-zywina}
Let $E$ be a non-CM elliptic curve defined over $\Q$. Let $j_E$ be the $j$-invariant of
$E$ and let $h(j_E)$ be its logarithmic height. Let $N$ be the product of primes for which $E$ has bad
reduction.
There are absolute constants $C$ and $\gamma$
such that
$$[\mathrm{GL}_2(\hat{\Z}) : \rho_{E}(G_{\mathbb{Q}}) ] \leq C \max{(1, h(j_E))}^\gamma.$$
\end{theorem}

\subsection{Serre curves} \label{serre-curves}

From Serre's theorem, we know that
there exist  positive integers $m$ so that, if
$$
\pi : \mathrm{GL}_{2}(\hat{\mathbb{Z}}) \longrightarrow \mathrm{GL}_{2}\left(\mathbb{Z}/ m \mathbb{Z} \right)
$$
is the natural projection, we have
\begin{equation} \label{defME}
\rho_{E}\left( G_{\mathbb{Q}}\right) = \pi^{-1} \left(  G_{E}(m)  \right) ,
\end{equation}
i.e., $\rho_{E}\left( G_{\mathbb{Q}}\right)$ is the full inverse image of $G_{E}(m)$.
For a non-CM curve $E$ over $\mathbb{Q}$, let us denote by $M_{E}$ the smallest positive integer $m$ such that
\eqref{defME} holds.
Then, $M_E$ has the following properties:
\begin{eqnarray}
\label{serre1} &&\mbox{If $(n, M_E)=1$, then $G_E(n) =
\mbox{GL}_2(\Z/n \Z)$;}
\\
\label{serre2} &&\mbox{If $(n, M_E)=(n,m)=1$, then $G_E(mn) \simeq
G_E(m) \times G_E(n)$;}
\\
\label{serre3} && \mbox{If $M_E \mid m$, then $G_E(m) \subseteq \mbox{GL}_2(\Z/m \Z)$
is the
full inverse image of} \\ \nonumber
&& \mbox{$G_E(M_E) \subseteq \mbox{GL}_2(\Z/M_E \Z)$ under the projection  map.}
\end{eqnarray}

Serre \cite{Serre} observed that, although $\rho_{E}(G_\Q)$ has finite index in $\mathrm{GL}_{2}(\hat{\mathbb{Z}})$, it is never surjective when the base field is
$\mathbb{Q}$.
Indeed,
suppose that an elliptic curve $E$ is given by the Weierstrass equation
$$y^2 = (x - e_{1}) (x - e_{2}) (x - e_{3}).
$$
 Then, the $2$-torsion of $E$ can be expressed explicitly as
$$
E[2] = \{ \mathcal{O}, (e_{1}, 0), (e_{2}, 0), (e_{3}, 0)\}.
$$
The discriminant $\Delta_{E} $  of $E$ is defined as follows:
$$
\Delta_{E} = (e_{1} - e_{2})^2 (e_{2} - e_{3})^2 (e_{3} - e_{1})^2.
$$
The definitions of $E[2]$ and $\Delta_{E}$ immediately imply that
$$
\mathbb{Q}\left(\sqrt{\Delta_{E}}\right) \subseteq \mathbb{Q} \left(E[2]\right),
$$
and $\rho_E$ is not surjective.

In fact, for each
elliptic curve $E$ over $\Q$, there is an index two subgroup $H_E \subseteq  \gl(\hat{\Z})$
such that
$$\rho_E(G_\Q) \subseteq H_{E} \subseteq \mathrm{GL}_{2}(\hat{\mathbb{Z}}).$$
For a precise
definition of $H_E$, we refer the reader to the original paper of Serre \cite{Serre},
or the nice exposition in \cite[Section 4]{jonesthesis}.

With this in mind, we can state the following definition:
\begin{definition}
An elliptic curve $E$ over $\mathbb{Q}$ is a Serre curve if $\rho_{E}\left( G_{\mathbb{Q}}\right) = H_{E}$.
\end{definition}

Throughout this paper, let $\mathcal{N}(A, B)$ denote the non-Serre curves in $\mathcal{C}(A, B)$ and let $\mathcal{S}(A,B)$ denote the set of Serre curves. Then, we certainly have $\mathcal{C}(A, B) = \mathcal{S}(A, B) \cup \mathcal{N}(A, B).$ This decomposition will be useful as it enables us to take separate averages over Serre versus non-Serre curves.

Jones showed in \cite{jonesthesis} that most elliptic curves over $\Q$ are Serre curves.
In our situation, his result can be stated as follows:

\begin{theorem} \cite[Theorem 25]{jones}
\label{thm-jones}
There is an absolute constant $\beta > 0$ such that
$$
\frac{|\mathcal{N}(A,B)|}{|\mathcal{C}(A,B)|} \ll \frac{\log^\beta(\min(A,B))}{\sqrt{\min(A,B)}}.$$
\end{theorem}

\subsection{Effective Chebotarev Density Theorem} \label{section-CDT}

Let $K/\Q$ be a finite Galois extension with Galois group ${\rm Gal}\left(K/\Q\right)$, and let $C$ be a union of conjugacy classes in ${\rm Gal}\left(K/\Q\right)$. Let $n_K$ be the degree of $K/\Q$, and let $d_K$ be an absolute discriminant of $K$. Let $\mathcal{P}(K)$ be the set of ramified primes, and let
$$
m_K = n_K \prod_{p \in \mathcal{P}(K)}p.
$$

If $\phi_p : {\rm Gal}(\overline{\Q}_p/\Q_p) \rightarrow {\rm Gal}(\overline{\F}_p/\F_p)$ is the Frobenius map given by $\phi_p: x \mapsto x^p$, we define $\sigma_p$ to be the pullback of $\phi_p$. If $p \nmid d_K$, for each unramified prime $p$, $\sigma_p$ is the Artin symbol at the prme $p$, which is
well-defined up to conjugation.
Let
$C$ be a union of conjugacy classes in ${\rm Gal}\left(K/\Q\right)$. Let
$$\pi_{C} (X, K) = \#\{p\leq X : p\nmid d_K \ \hbox{and} \ \sigma_p \in C \}.$$

The following theorem is an effective version of the Chebotarev Density Theorem due to Lagarias and Odlyzko \cite{lo}, with a refinement due to Serre \cite{Serre1}.

\begin{theorem}\label{effectivecdt}
\begin{enumerate}[(i)]
\item{Let $\beta$ be the exceptional zero of the Dedekind zeta function associated to $K$ (if such a zero exists). Then, for all $X$ such that $$\log X\gg n_K(\log d_K)^2,$$ we have that
\begin{eqnarray*}
\pi_{C}(X, K) &=& \frac{|C|}{|{\rm Gal}(K/\Q)|} \pi(X) \\
&&  +O\left(\frac{|C|}{|{\rm Gal}(K/\Q)|}\pi(X^\beta)+|\widetilde{C} |X\cdot \exp\Big(-\frac{c}{\sqrt{n_K}} \sqrt{\log X}\Big)\right),\end{eqnarray*}
where $c$ is a positive absolute constant and $|\widetilde{C}|$ is the number of conjugacy classes in $C$.}\label{item1}
\item{Assuming the GRH for the Dedekind zeta function of $K$, we have that
$$\pi_{C}(X, K) = \frac{|C|}{|{\rm Gal}(K/\Q)|} \pi(X)+O\left(\sqrt X|C| \log(m_K X)\right).$$}
\end{enumerate}
\end{theorem}

We will make use of the unconditional bound given in Theorem \ref{effectivecdt}(\ref{item1}) in our proof of Theorem \ref{upperbound}. We need the following lemmas to make the error term explicit.

\begin{lemma} \cite{stark} \label{stbound} Let $K/\Q$ be a finite Galois extension of degree $n_K$ and discriminant $d_K$.
Then, for the exceptional zero $\beta$ of the Dedekind zeta function associated to $K$, we have
\begin{equation}\label{beta} \beta < 1 - \frac{A_1}{\max\{|d_K|^{1/n_K}, \log|d_K|\}},\end{equation} where $A_1$ is a positive constant.
\end{lemma}

\begin{lemma} \cite[Proposition 6, Section 1.4]{Serre1} \label{sebound} Let $K/\Q$ be a finite Galois extension of degree $n_K$ and discriminant $d_K$. Let $\mathcal{P}(K)$ be the set of ramified primes.
Then, $$\frac{n_K}{2} \sum_{p \in \mathcal{P}(K)} \log{p} \leq \log d_K \leq (n_K-1)
\sum_{p \in \mathcal{P}(K)} \log{p}  + n_K \log{n_K}.$$ \end{lemma}

\begin{comment}
Throughout this paper, we take $K = \Q(E[n]).$ In the case of non-CM elliptic curves, we have (cf. \cite[(1.3)]{cojocaru2}) that \begin{equation}\label{Kdegree}[\Q(E[n]): \Q] = \#GL_2(\Z/n\Z) = n^4\prod_{\substack{q \mid n \\ q \ \mathrm{prime}}}\left(1 - \frac{1}{q}\right)\left(1 - \frac{1}{q^2}\right).\end{equation} From elementary number theory, we have $$\prod_{\substack{q \mid n \\ q \ \mathrm{prime}}}\left(1 - \frac{1}{q}\right)\left(1 - \frac{1}{q^2}\right) \asymp \frac{\varphi(n)}{n}.$$ As a result, we may assume that $n_K = [\Q(E[n]): \Q] \leq n^4$. Using this bound for $n_K$ in conjunction with the bounds given by \eqref{beta} and Lemma \ref{dkbound} allows us to deduce the following corollary of Theorem \ref{effectivecdt}(\ref{item1}).
\end{comment}

\begin{corollary}\label{usefulcdt} Let $K = \Q(E[n])$, and $C$ a union of conjugacy classes
in ${\rm Gal}(K/\Q).$ For all $X$ such that $\log X \gg_E n^{12}(\log n)^2$, we have $$\pi_{{C}}(X, K) = \frac{|C|}{|{\rm Gal}(K/\Q)|} \pi(X) + O\left(X\exp\left(-\frac{A}{n^2} \sqrt{\log X} \right)\right),$$ where $A$ is an absolute constant. \end{corollary}
\begin{proof}
This follows immediately from using the bounds given in Lemmas \ref{sebound} and \ref{stbound}
in Theorem \ref{effectivecdt}(\ref{item1}): for $K=\Q(E[n])$, we have that $n_K \leq \# \gl(\Z/n\Z) \leq n^4$ and
$\log{d_K} \ll n^4 \log(n N_E),$ where $N_E$ is the conductor of $E$. We can apply Theorem \ref{effectivecdt}(\ref{item1}) when
$\log X \gg n^{12}(\log N_E n)^2$.
\end{proof}

We conclude this section by explaining how the preceding corollary is related to $\pi_{E,f}^{SF}(X)$.
Let $p \nmid n N_E$, which implies that $p$ is unramified in $K = \Q(E[n])$.
Since the Frobenius endomorphism
$(x, y) \mapsto (x^p, y^p)$ of the reduction of $E$ over the finite field $\F_p$ satisfies the polynomial $x^2 -  a_p(E) x + p$, it follows from the definition of the Frobenius element $\sigma_p$ that $\rho_{E,n}(\sigma_p)$ must have characteristic polynomial $x^2 -  a_p(E) x + p$ in
$\gl(\Z/n\Z)$; i.e., we must have
\begin{eqnarray*}
\tr\rho_{E,n}(\sigma_p) &\equiv& a_p(E) \mod{n} \\
\det\rho_{E,n}(\sigma_p) &\equiv& p \mod{n}.
\end{eqnarray*}
Thus, since $f_p(E) := f(a_p(E), p)$, we have that
\begin{eqnarray*}
\# \left\{ p \leq X : f_p(E) \equiv 0 \mod n \right\} &=&
\# \left\{ p \leq X : f(\tr\rho_{E,n}(\sigma_p) , \det\rho_{E,n}(\sigma_p) ) \equiv 0 \mod n \right\} \\
&=& \# \left\{ p \leq X : \sigma_p \in C_{E,f}(n) \right\}
\end{eqnarray*}
where $C_{E,f}(n)$ is the union of conjugacy classes defined by
\eqref{def-cef}.

\begin{comment}
Thus, in order to find an upper bound for $\pi_{\mathcal{C}}(X, K)$, it remains for us to estimate the size of $\frac{|\mathcal C|}{|{\rm Gal}(K/\Q)|}$. We remark that the map $\rho_{E,n}: G(n) \to GL_2(\Z/n\Z)$ defined in Section \ref{torsionfields} has the following properties: \begin{enumerate}[(i)] \item{$\tr(\rho_n(\sigma_p)) \equiv a_p \mod{n}$;} \item{$\det(\rho_n(\sigma_p)) \equiv p \mod{n}$}. \end{enumerate} This reduces the task of computing $\frac{|\mathcal C|}{|{\rm Gal}(K/\Q)|}$ to a problem of counting matrices with a given trace and determinant. We will demonstrate precisely how these counts can be obtained in the following two sections.
\end{comment}

%----------------------------------------------------------------------------
\section{Key Lemma}
%-------------------------------------------------------------------------

\begin{lemma} \label{keylemma} Let $f(x,y)$ be any non-constant squarefree polynomial in $\Z[x,y]$.
Then, for any $\varepsilon > 0$ and any squarefree integer $n$, we have
\begin{eqnarray} \label{upperbounds-n1}
\frac{|C_{f}(n^2)|}{|\mathrm{GL}_2(\Z / n^2 \Z)|} \ll_f \frac{1}{n^{2 - \varepsilon}}.
\end{eqnarray}
\end{lemma}

\begin{proof}

We begin by showing that for any prime $p$, we have
\begin{eqnarray}
\label{numbermat}
|C_{f}(p^2)| = \# \left\{ g \in \gl(\Z/p^2 \Z) \;:\; f(\tr{g}, \det{g}) \equiv 0 \mod p^2 \right\} \ll_f
{p^{6}}. \end{eqnarray}
%%Since $n$ is squarefree, it is sufficient
%% to evaluate $|C_{f}(n)|$ for each odd $p$ dividing $n$,
%%%$$|N_f(p)| :=  \# \left\{ g \in \gl( \Z / p \Z) \; : \; f(\tr{g}, \det{g}) \equiv 0 \mod p \right\};$$
%%%and the bound for $|C_{f}(n)|$ will follow from applying the Chinese Remainder Theorem.
Let \[ \left( \begin{array}{cc}
a & b  \\
c & d \end{array} \right) \in \gl(\Z / p\Z).\] For each pair $(D,T)$ with $D \in \F_p^*$ and $T \in \F_p$, we first count the matrices in $\gl( \Z / p \Z)$ with determinant $ad - bc = D$ and trace $a + d = T$. We consider the following two cases:

\noindent \textbf{Case 1}: $ad - D \not\equiv 0 \pmod{p}$.

We observe that $ad - D = (T - d)d - D \equiv 0 \pmod{p}$ if and only if $d^2 - Td + D \equiv 0 \pmod{p}$. This criterion is satisfied for $N := 1 + \left(\frac{T^2 - 4D}{p}\right)$ values of $d$, where $\left(\frac{\cdot}{p}\right)$ is the Legendre symbol. Thus, the number of values of $d$ in $\gl(\Z/p\Z)$ for which $ad - D \not\equiv 0 \pmod{p}$ is $p - N.$ The choice of $a$ is completely determined by the choice of $d$. Moreover, the number of choices for the pair $(b,c)$ is $p-1$, since we must exclude the pair that would yield $ad - D \equiv 0 \pmod{p}$. As a result, we have $(p-N)(p-1)$ matrices with the prescribed properties.

\noindent \textbf{Case 2}: $ad - D \equiv 0 \pmod{p}.$

From the previous case, we see that the number of choices for $d$ is $N$ and the number of choices for $a$ is $1$. In this case, we have $2p - 1$ choices for $b$ and $c$. This gives us $(2p - 1)N$ matrices with $ad - D \equiv 0 \pmod{p}.$

By summing the counts obtained in the two cases described above, we see that the full count of matrices in $\gl(\Z/p\Z)$ with determinant $D$ and trace $T$ is $$(p-N)(p-1) + (2p-1)N = p^2 + p(N-1) = p^2 + O(p).$$

Therefore, letting $S_{f,D}(p)$ be the set of roots of the polynomial $f(x, D)$ over $\F_p$ for any
$D \in \F_p^*$, we have that
\begin{eqnarray*}
|C_f(p)| &=& \sum_{D \in \F_p^*} \# \left\{ g \in \gl( \Z / p \Z) \; : \; f(\tr{g}, D) =0 \right\} \\
&\leq& \sum_{D \in \F_p^*, \ T \in S_{f,D}(p)} \# \left\{ g \in \gl( \Z / p \Z) \; : \; \tr{g}=T,
\ \det{g} = D \right\}  \\
&\ll& \sum_{D \in \F_p^*} |S_{f,D}(p)| \, p^2 \leq (\deg_{x}{f}) \cdot p^3 \ll_f p^3.
\end{eqnarray*}
%%Then, applying the Chinese Remainder Theorem over all odd prime divisors of $n$ yields
%% \begin{eqnarray} \label{applyCRT}
%% |C_f(n)|  \ll \prod_{p \mid n} (\deg_{x}{f})\cdot p^3 \ll  (\deg_{x}{f})^{\omega(n)} n^3 \ll_f %%n^{3+\epsilon},
%% \end{eqnarray}
%%which proves \eqref{numbermat}. Since
%%$$|\gl(\Z/n\Z)| = \prod_{p \mid n} p (p^2-1) (p-1) = n^4 \prod_{p \mid n} \frac{p-1}{p} %%\frac{p^2-1}{p^2} \gg n^4 \frac{\varphi(n)}{n},$$
%%this proves that
%%\begin{eqnarray} \label{one}
%%\frac{|C_{f}(n)|}{|\gl(\Z/n\Z)|} \ll_f \frac{1}{n^{1-\varepsilon}}
%%\end{eqnarray}
%%for any $\varepsilon > 0$.

%We now proceed to show that
%\begin{eqnarray} \label{two}
%\frac{|C_{f}(n^2)|}{|\gl(\Z/n^2\Z)|} \ll_f  \frac{1}{n^{1-\varepsilon}}
%\frac{|C_{f}(n)|}{|\gl(\Z/n\Z)|}.
%\end{eqnarray}
%and the lemma follows from \eqref{one} and \eqref{two}.
%By counting the lifts $\tilde{g} \in \gl(\Z / n^2 \Z)$ of matrices $$g =
%\left( \begin{array}{ll} a & b \\ c & d \end{array} \right)
%\in \gl(\Z / n \Z),$$ namely,
%$$
%\tilde{g} = \left( \begin{array}{ll} a + k_1 n & b + k_2 n \\ c + k_3 n & d + k_4 n \end{array} \right), \quad
%1 \leq k_i \leq n, \quad i=1,2,3,4,
%$$
%we obtain that $|G_E(n^2)| = n^4 |G_E(n)|$ (because
%$g$ is invertible in $M_2(n)$ if and only if
%$\tilde{g}$ is invertible in $M_2(n^2)$).

Then, in order to bound ${|C_{f}(p^2)|}$, we want to count
of lifts $\tilde{g} \in \gl(\Z/p^2 \Z)$ of a given matrix $g \in C_f(p)$ which satisfy
\begin{eqnarray} \label{newcong}
f(\tr{\tilde{g}}, \det{\tilde{g}}) \equiv 0 \mod p^2.\end{eqnarray}
We write
$$
\tilde{g} = \left( \begin{array}{ll} a + k_1 p & b + k_2 p \\ c + k_3 p & d + k_4 p \end{array} \right), \quad
1 \leq k_i \leq p, \quad i=1,2,3,4,
$$
and  $T= \tr{g}, D = \det{g}, \tr{\tilde{g}} = T + p u,
\det{\tilde{g}} = D + p v$. Using the Taylor expansion of $f$, we have that
$$f(T+pu, D+pv) \equiv  f(T,D) + p \left( u \frac{\partial f}{\partial x}(T,D) +
v \frac{\partial f}{\partial y}(T,D) \right) \mod p^2.$$
Let
\begin{eqnarray*}
h(k_1, k_2, k_3, k_4) &=& \left( u \frac{\partial f}{\partial x}(T,D) +
v \frac{\partial f}{\partial y}(T,D) \right) \\ &=&
\left( d \frac{\partial f}{\partial y}(T,D) + \frac{\partial f}{\partial x}(T,D) \right)k_1 +
\left( a \frac{\partial f}{\partial y}(T,D) + \frac{\partial f}{\partial x}(T,D) \right) k_4 \\
&& \quad - b \frac{\partial f}{\partial y}(T,D) k_3 - c \frac{\partial f}{\partial y}(T,D) k_2.
\end{eqnarray*}
Then,
we need to count the number of solutions to the congruence
\begin{eqnarray}\label{hcong}
h(k_1, k_2, k_3, k_4) \equiv - \frac{f(T,D)}{p} \mod p.
\end{eqnarray}
(Recall that $p$ divides $f(T,D)$ by hypothesis, since we are lifting elements of $C_f(p)$).

If $h(k_1, k_2, k_3, k_4) \neq 0$, the number of solutions $(k_1, k_2, k_3, k_4)$ to the congruence given in \eqref{hcong} is bounded by $\ll_f p^3$.
If  $h(k_1, k_2, k_3, k_4) = 0$, then we can have $p^4$ solutions $(k_1, k_2, k_3, k_4)$ if $f(T,D) \equiv 0
\mod p^2$. Notice that, unless $b=c = 0$, we have that $h(k_1, k_2, k_3, k_4) \neq 0$,
except in the case where $$\frac{\partial f}{\partial x}(T,D) = \frac{\partial f}{\partial y}(T,D) \equiv  0 \mod p.$$
So, we only need to consider the pairs $(T,D)$ such that
\begin{eqnarray} \label{number-sol}
f(T,D) = \frac{\partial f}{\partial x}(T,D) = \frac{\partial f}{\partial y}(T,D) \equiv  0 \mod p.
\end{eqnarray}
We claim there is a bounded number of such pairs $(T,D)$ when $f(x,y)$ is squarefree. Indeed, in that
case $f(x,y)$ and $\displaystyle \frac{\partial f}{\partial x}$ are co-prime, and it follows from
the polynomial analogue of Bezout's identity (Max Noether's fundamental theorem \cite[p.702]{GH})
%%% Griffiths and Harris, p. 702
that
one can find polynomials $a(x,y), b(x,y) \in \Z[x,y]$ and  $\Delta_1(x) \in \Z[x]$ such that
$$
a(x,y) f(x,y) + b(x,y) \frac{\partial f}{\partial x}(x,y) = \Delta_1(x).
$$
Similarly, one can find polynomials
$a(x,y), b(x,y) \in \Z[x,y]$ and $\Delta_2(y) \in \Z[y]$ such that
$$
a(x,y) f(x,y) + b(x,y) \frac{\partial f}{\partial y}(x,y) = \Delta_2(y).
$$
Then, the number of $(T,D) \in \F_p^2$ satisfying (\ref{number-sol}) is bounded by
$\deg \Delta_1 \times \deg \Delta_2$, independently of $p$.

Thus, we see that each matrix in $C_f(p)$ lifts to either $\ll_f p^3$ matrices or $\ll_f p^4$ matrices (in the case where $h(k_1, k_2, k_3, k_4) = 0)$. So, for each prime $p$, we  have
$$
|C_f(p^2)| \ll_f p^6,$$ which proves (\ref{numbermat}). It follows immediately
that
$$
\frac{|C_f(p^2)|}{|\gl(\Z / p^2 \Z)|} \ll_f \frac{1}{p^2}.
$$

Finally, by applying the Chinese Remainder Theorem over all prime divisors of the
squarefree integer $n$, we have
that
$$
\frac{|C_f(n^2)|}{|\gl(\Z / n^2 \Z)|} = \prod_{p \mid n} \frac{|C_f(p^2)|}{|\gl(\Z / p^2 \Z)|}
\ll_f \prod_{p \mid n} \frac{1}{p^2} \ll \frac{1}{n^{2-\varepsilon}},
$$
which concludes the proof of the lemma.

%%\ll_f \prod_{p \mid n} p^6 \ll_f n^{6+\epsilon}.
%%$$ Since $$|\gl(\Z/n^2\Z)| = \prod_{p \mid n} p^2(p^4-1)(p^2-1) = n^8 \prod_{p \mid n} %%\frac{p^4-1}{p^4} \frac{p^2-1}{p^2},$$ we have $$\frac{|C_f(n^2)|}{|\gl(\Z/n^2\Z)|} \ll_f %%\frac{1}{n^{2-\varepsilon}},$$ which concludes the proof of our theorem.
\end{proof}

%Therefore, we have
%$$
%\frac{|C_f(p^2)|}{|\gl(\Z / p^2 \Z)|} \ll_f \frac{1}{p^2}.$$

%One could also work on the fact that we want $f(t,d) \equiv 0 \mod p^2$, but I got confused there.
%This does not imply for example
%$$ \frac{\partial f}{\partial x}(t,d) \equiv \frac{\partial f}{\partial x}(t,d) \equiv 0 \mod p .$$

%Also, maybe it is better to write the proof of this lemma prime by prime: First prove that
%$$
%\frac{|C_f(p^2)|}{|\gl(\Z / p^2 \Z)|} \ll \frac{1}{p^{2}}
%$$
%which implies by the Chinese Remainder Theorem that
%$$
%\frac{|C_f(n^2)|}{|\gl(\Z / n^2 \Z)|} \ll \frac{1}{n^{2-\epsilon}}.
%$$
%We do it for the first part, we should also do it for the lifts.

%*************************************

%As in the squarefree case, we can use the Chinese Remainder Theorem over all odd prime factors of $n$ to show that each $g \in C_{f}(n)$ lifts to at most $n^{3 + \varepsilon}$
%matrices $\tilde{g} \in C_{f}(n^2)$. Then,
%$$
%\frac{|C_f(n^2)|}{|\gl(\Z/ n^2\Z)|} \ll_f \frac{n^{3 + \epsilon}}{n^4} \frac{|C_f(n)|}{|\gl(\Z/ n\Z)|},
%$$
%and the bound given in \eqref{two} holds.

%----------------------------------------------------------------------------
\section{Proof of Theorem \ref{upperbound}}\label{ub}
%-------------------------------------------------------------------------

Our proof of Theorem \ref{upperbound} will rely on the following lemma:

\begin{lemma}\label{inclusion-exclusion} Let $C_{E,f}^{SF}$ be the conjectural constant defined by \eqref{defCE}. Then
$$C_{E,f}^{SF} = \sum_{d=1}^\infty \mu(d)\frac{|C_{E,f}(d^2)|}{|G_E(d^2)|}.$$
\end{lemma}
\begin{proof}
By the properties \eqref{serre1} and \eqref{serre2} of $M_E$ and the Chinese Remainder Theorem, we can write
\begin{eqnarray*}
\sum_{d=1}^\infty \mu(d)\frac{|C_{E,f}(d^2)|}{|G_E(d^2)|} &=&
\sum_{k \mid M_E} \sum_{\substack{ d=1 \\(d, M_E)=k}}^\infty \mu(d)\frac{|C_{E,f}(d^2)|}{|G_E(d^2)|} \\
&=& \sum_{k \mid M_E}  \mu(k)\frac{|C_{E,f}(k^2)|}{|G_E(k^2)|}
\sum_{\substack{j=1 \\(j, M_E)=1}}^\infty \mu(j)\frac{|C_{E,f}(j^2)|}{|G_E(j^2)|} \\
&=& \sum_{k \mid M_E}  \mu(k)\frac{|C_{E,f}(k^2)|}{|G_E(k^2)|}
\prod_{\ell\nmid M_E} \left(1-\frac{|C_{f}(\ell^2)|}{|\mbox{GL}_2(\Z/\ell^2\Z)|}\right) = C_{E,f}^{SF} . \end{eqnarray*}
\end{proof}

Now we commence with our proof of Theorem \ref{upperbound}. For every real number $z \geq 2$, we have
	$$\pi_{E,f}^{\rm SF}(X) \leq \# \left\{p\leq X\;|\; \ell^2\nmid f_p(E), \; \forall \ell\le z\right\}.$$ Let $P(z) := \prod_{\ell\leq z} \ell,$ and define
	$$\Omega_E(P(z)^2) := \left\{g\in G_E(P(z)^2)\;|\; \ell^2\nmid f(\tr{g}, \det{g}), \; \forall \ell\leq z\right\}.$$
	
	 Moreover, let $n=P(z)^2$ and $K= \Q(E[n]).$ As described at the end of Section \ref{section-CDT}, we can use Corollary \ref{usefulcdt} to obtain
\begin{eqnarray*}
 \# \left\{p\leq X\;|\; \ell^2\nmid f_p(E), \; \forall \ell\le z\right\}
	&=&  \# \left\{p\leq X\;|\; \sigma_p \in \Omega_E(P(z)^2) \right\} \\
	&=& \pi(X)\cdot \left|\frac{\Omega_E(P(z)^2)}{G_E(P(z)^2)}\right|+ O\left(X\exp\left(-\frac{A}{P(z)^4} \sqrt{\log X} \right)\right),
	\end{eqnarray*}
for $X$ sufficiently large (where $A$ is an absolute constant). Taking $\log X\gg P(z)^{24} (\log P(z))^2$ yields
	$$P(z) \ll_E \log^{\frac{1}{24}-\varepsilon} X,$$  for any $\varepsilon > 0$.
Then our error term is
	$$O\left(X\exp \left(-\frac{A}{P(z)^4}\sqrt{\log X}\right)\right) = O_E \left(X\exp\left(-A
(\log{X})^{1/3+\varepsilon}\right)\right).$$

Now, using Lemma \ref{inclusion-exclusion}, we obtain
%%Note: We previously had "by the Chinese Remainder Theorem" written above, but I think it's really just inclusion-exclusion.
	\begin{align*}
	\frac{|\Omega_E(P(z)^2)|}{|G_E(P(z)^2)|} & = \sum_{n\mid P(z)} \mu(n)\frac{|C_{E,f}(n^2)|}{|G_E(n^2)|} \\
	&=C_{E,f}^{\rm SF} + O\left(\sum_{n\geq z}\frac{|C_{E,f}(n^2)|}{|G_E(n^2)|}\right).
\end{align*}

Proceeding as in the proof of Lemma \ref{inclusion-exclusion}, we have that
\begin{eqnarray*}
\sum_{n\geq z}\frac{|C_{E,f}(n^2)|}{|G_E(n^2)|} &\leq& \sum_{k \mid M_E} \frac{|C_{E,f}(k^2)|}{|G_E(k^2)|}
\sum_{j\geq z/k}\frac{|C_{f}(j^2)|}{|\mbox{GL}_2(\Z/j^2\Z)|} \\
&\ll_E& \sum_{j\geq z/M_E}\frac{|C_{f}(j^2)|}{|\mbox{GL}_2(\Z/j^2\Z)|} \\
&\ll_{E, f} & \sum_{j \geq z/M_E} \frac{1}{j^{2-\varepsilon}} \ll_{E, f} \frac{1}{z^{1-\varepsilon}},
\end{eqnarray*}
where the penultimate inequality follows from Lemma \ref{upperbounds-n1}.

Therefore, we have
	$$\pi_{E,f}^{\rm SF}(X) \leq C_{E, f}^{\rm SF}\cdot\pi(X)+O_{E,f} \left(\frac{\pi(X)} {z^{1-\varepsilon}} + X\exp\left(-(\log X)^{{1}/{3}+\varepsilon}\right)\right).$$ To optimize, we want to choose the largest possible value of $z$ such that $P(z) \ll \log^{\frac{1}{24} - \varepsilon} X$. We take $z = c \log \log X$ for $c>0$ small enough, which yields
	$$\pi_{E,f}^{\rm SF}(X) \leq C_{E, f}^{SF}\cdot \pi(X) \left(1+O_{E,f} \left(\frac{1}{(\log \log X)^{1 -\varepsilon}}\right)\right).$$
This completes the proof of Theorem \ref{upperbound}.

%\textbf{Note: This only partially solves our problem. We have only bounded $\pi_{E,f}^{SF}(X,z)$ when, in fact, we wanted to bound $\pi_{E, f}^{SF}(X)$. In order to complete the claim stated in Theorem 1.1, we will need to handle the case where $\ell > z.$ I think it will come down to something similar to the argument that $$\#\{n \leq X: n \ \hbox{not squarefree but satisfies} \ \ell^2 \nmid n \ \hbox{for all} \ \ell \leq z\} = o(X).$$}

%\textbf{Shabnam's Answer to Lola's Note above} I may not have understood the question well. But I think there is no such thing as $\pi_{E,f}^{SF}(X)$ independent of $f$. According to the definition of $\pi_E^{SF}(X)$ there is always a function $f$ involved. Here since our $f$ is fixed we just did not write it anymore! I maybe wrong!

%\textbf{Lola's Response to Shabnam's Answer}: You're correct that $\pi_E^{SF}(X)$ is not independent of $f$. It's just the notation that we used to indicate the count of primes $p \leq X$ for which $p + 1 - a_p$ is divisible by a square of a small prime (where by ``small'', I mean that $\ell \leq z.$) We never handled the case where $\ell > z$, which needs to be taken into account in our bound for $\pi_{E, f}^{SF}.$ I agree that the notation is suboptimal, since there is still a dependence on $f$. Maybe instead we should use $\pi_E^{SF}(X, z)$, since $z$ is really the parameter that is changing. In other words, $$\pi_E^{SF}(X, z) : = \#\{p \leq X: \ell^2\nmid (p+1-a_p), \; \forall \ell\le z\}.$$ Would everyone be ok with that?

%-------------------------------------------------------------
\section{Averaging the constants over families of elliptic curves}\label{avgconstants}
%---------------------------------------------------------------

%%We give in this section the proof of Theorem \ref{***}.
%%In this section, we show that
%%\begin{eqnarray} \label{averageCy}
%%\frac{1}{|\mathcal{C}(A, B)|} \sum_{E \in \mathcal{C}(A, B)} C_{E,f}^{SF} \sim C_{f}^{SF}
%%\end{eqnarray}
%%when $A,B \rightarrow \infty$. %(Proposition \ref{lemma-nonserre} and Propostion %%\ref{average-serre-curves}).

%Then, we have
%\begin{eqnarray*}
%\frac{1}{|\mathcal{C}(A, B)|} \sum_{E \in \mathcal{C}(A, B)}\sum_{d \leq y} \mu(d)
%\frac{|C_{E,f}(d^2)|}{|G_E(d^2)|}
%&\sim&  \frac{1}{|\mathcal{C}(A, B)|} \sum_{E \in \mathcal{C}(A, B)} C_{f,y}^{SF}\\
%&=& \frac{1}{|\mathcal{C}(A, B)|} \sum_{E \in \mathcal{C}(A, B)} \left( C_{f}^{SF} - \sum_{d > y} \mu(d)
%\frac{|C_{f}(d^2)|}{|\gl(d^2)|} \right) \\
%&=& C_{f}^{SF} +  \frac{1}{|\mathcal{C}(A, B)|} \sum_{E \in \mathcal{C}(A, B)} O_f \left( \sum_{d \geq y} \frac{1}{d^{2-\varepsilon}} \right)\\
%&=& C_{f}^{SF} +  O_f \left( \frac{1}{y} \right)
%\end{eqnarray*}
%where the bound $$\frac{|C_{f}(d^2)|}{|\gl(d^2)|} \ll \frac{1}{d^{2-\varepsilon}}$$
%follows from Lemma \ref{boundratio}.

In this section, we prove Theorem \ref{thmaveragec} by
separating the family of curves $E \in \mathcal{C}$ into two subsets: Serre curves and non-Serre curves. We handle the average over non-Serre curves in Section \ref{nonserre}, and we compute the average over
Serre curves in Section \ref{serre}.

%-------------------------------------------------------------
\subsection{Averaging over non-Serre curves}\label{nonserre}
%---------------------------------------------------------------

%We prove in this section that
\begin{proposition} \label{lemma-nonserre} There exists an absolute constant $\delta > 0$ such that
$$\frac{1}{| \mathcal{C}(A,B)|} \sum_{E \in \mathcal{N}(A,B)} C_{E,f}^{SF} \ll \frac{\log^{\delta}(AB)}{\sqrt{\min{(A,B)}}}.$$
\end{proposition}
\begin{proof}
For any $E \in \mathcal{C}(A,B)$,
we have that
\begin{eqnarray*}
C_{E,f}^{SF} &=& \sum_{d=1}^\infty \mu(d) \frac{|C_{E,f}(d^2)|}{|G_E(d^2)|} \\ &\leq&
\sum_{d=1}^\infty \frac{|C_{f}(d^2)|}{|G_E(d^2)|} \\
&\leq& [\mathrm{GL}_2(\hat{\Z}) : \rho_E(G_\Q)]  \sum_{d=1}^\infty \frac{|C_{f}(d^2)|}{|\gl(\Z/d^2\Z)|}\\
&\ll&  [\mathrm{GL}_2(\hat{\Z}) : \rho_E(G_\Q)]
\end{eqnarray*}
where the final inequality follows from Lemma \ref{keylemma}.

Using Theorem \ref{explicitform-zywina}, we have that for any $E(a,b) \in \mathcal{C}(a,b)$,
$$
C_{E,f}^{SF} \ll  [\mathrm{GL}_2(\hat{\Z}) : \rho_E(G_\Q)] \ll
( \max{(1, h(j_{E(a,b)}))} )^{\gamma}
$$
where $\gamma$ is an absolute constant. Since $|a| \leq A$ and $|b| \leq B$, we have that
\begin{eqnarray*}
h(j_{E(a,b)}) &=& h \left( \left[ 1728(4a)^3, - 16 (4a^3 + 27b^2) \right] \right) \\&\ll&
\log{(\max{(A, B)} )} \leq \log{AB},
\end{eqnarray*}
and then $C_{E(a,b), f}^{SF} \ll (\log{AB})^\gamma.$ Now, using Theorem  \ref{thm-jones} to bound
the size of $\mathcal{N}(A,B)$, we get immediately that
$$
\frac{1}{|\mathcal{C}(A,B)|} \sum_{E \in \mathcal{N}(A,B)} C_{E,f}^{SF} \ll \frac{\log^{\beta+\gamma}(AB)}{\sqrt{\min{(A,B)}}},
$$
and Proposition \ref{lemma-nonserre} follows by taking $\delta = \beta + \gamma$.

\end{proof}

%---------------------------------------------------------------------------------------------------------------------------------------
\subsection{Averaging over Serre curves}\label{serre}
%-------------------------------------------------------------------------------------------------------------------------------------

%An elliptic curve $E$ over $\mathbb{Q}$ is called a Serre curve if $\phi_{E}\left( G_{\mathbb{Q}}\right) = H_{E}$ (see section \ref{pre} for notation).
%In other words,  an elliptic curve over $\mathbb{Q}$ is a Serre curve if $\phi_{E}\left( G_{\mathbb{Q}}\right)$ has index two in $GL_{2}(\hat{\mathbb{Z}})$. Thus, a Serre curve is an elliptic curve whose torsion representation has an image which is as large as possible.

%%Lola: I think that the above paragraph can be omitted since we say the same things in the "preliminaries" section. In fact, I think that the following paragraph (starting with "Let $\Delta_{SF}(E)$ be the squarefree part..." and ending with the proof of Lemma 7.1) should be moved to "preliminaries" as well. Is that ok with you guys?

In this section, our goal is to show the following.

\begin{proposition} \label{average-serre-curves}
Let $\mathcal{C}(A , B)$ be the set of elliptic curves given by equations $y^2 = x^3 + ax + b$, with $ 4a^3 + 27b^2 \neq 0$ and $|a|\leq  A$ and $|b| \leq B$.  Let $\mathcal{S}(A, B) \subseteq \mathcal{C}(A,B)$ be the subset of Serre curves. Let $f \in \Z[x,y]$ be a non-constant squarefree polynomial.

Then, we have
$$
\frac{1}{|\mathcal{C}(A , B)|} \sum_{E \in \mathcal{S}(A , B)} \left|C_{E,f}^{ SF}  - C_{f}^{SF}\right| \ll  \frac{1}{A} + \left(\frac{\log B (\log A)^7}{B}\right).
$$
Consequently,
$$\frac{1}{|\mathcal{C}(A,B)|} \sum_{E \in \mathcal{S}(A,B)} C_{E,f}^{SF} \sim
C_{f}^{SF}$$
as $A,B \rightarrow \infty$.
\end{proposition}

First, we review several relevant properties of Serre curves; we refer the reader to \cite{jones} for details and proofs.
Let $E$ be a Serre curve and let $\Delta_{SF}(E)$ be the squarefree part of the discriminant of $E$. Note that $\Delta_{SF}(E)$ depends only on $E/\mathbb{Q}$, and not on the particular Weierstrass model. If $E$ is a Serre curve, then $\rho_E(G_\Q) = H_E$ (where $H_E$ is the subgroup of index $2$ defined in Section \ref{serre-curves}).
Also, we have that
\begin{equation}\label{E12Jones}
M_{E} = \left \{
\begin{array}{ll}
2 |\Delta_{SF}| \quad & \textrm{if $\Delta_{SF} = 1$ (mod \, $4$)} \\
4 |\Delta_{SF}| \quad & \textrm{otherwise},
\end{array}\right.
\end{equation}
and the subgroup $H_{E} = \rho_E(G_\Q)$ is the full pre-image of $G_E(M_E)$ under the canonical surjection
$$
\pi : \mathrm{GL}_{2}(\hat{\mathbb{Z}}) \longrightarrow \mathrm{GL}_{2}\left(\mathbb{Z}/ {M}_{E} \mathbb{Z} \right).
$$

Moreover, if $E$ is a Serre curve and $d \mid M_E,$ $d \neq M_E$, then
the natural projection of $G_E(M_E)$ into $\gl(\Z / d \Z)$ is surjective, i.e.,
\begin{eqnarray} \label{largeimage}
G_E(d)  = \gl(\Z / d \Z). \end{eqnarray}

When $E$ is a Serre curve, we can describe $G_E(M_E)$ explicitly by defining,
for each odd prime $p$,  the group homomorphisms
\begin{eqnarray*}
\psi_{p} : \gl (\Z/p\Z) &\rightarrow& \left\{ \pm 1 \right\}\\
g &\mapsto& \left( \frac{\det{g}}{p} \right) .
\end{eqnarray*}
We then define $\psi_{M_E}: \gl(\Z/M_E\Z) \rightarrow \left\{ \pm 1 \right\}$
by $$\psi_{M_E} ( \,\cdot\, ) = \psi_{2^{\nu_p(M_E)}} ( \,\cdot\, )
\prod_{p \parallel M_E} \psi_{p} ( \,\cdot\, ),$$
where the homomorphisms $\psi_{2^{k}}$ for $k =1,2,3$ are as described in \cite{jones}.
Then we have
$$G_E(M_E) = \psi_{M_E}^{-1}(1).$$

In order to prove Proposition \ref{average-serre-curves}, we will need the following pair of lemmas:

\begin{lemma} \label{whendifferent}
Let $E$ be an elliptic curve over $\mathbb{Q}$ which is a Serre curve. Let $n$ be
a squarefree integer such that $n \mid M_E$ and
$G_E(n^2) \neq  \mathrm{GL}_2(\Z / n^2 \Z)$. Then, either $n=M_E, n=M_E/2$ or $n=M_E/4$.
\end{lemma}

\begin{proof}
First, we assume that $E$ is a Serre curve, $n \mid M_E$, $n \neq M_E$ and
$(n, M_E/n)=1$. Under these assumptions, we have $n^2 \mid n M_{E}$ and
$n^2 \neq n M_{E}$.  The subgroup $G_E(n^2)$ of $\mbox{GL}_2(\Z /  n^2 \Z)$ is the
projection of $G_E(M_E n)$ obtained by reducing every matrix in $G_E(M_E n)$ modulo
$n^2$.  In order to prove that $G_E(n^2) = \mbox{GL}_2(\Z /  n^2 \Z)$, we will project $G_E(M_E n)$ into $\mbox{GL}_2(\Z /  n^2 \Z)$.
From
\eqref{serre3}, it follows that $G_E(M_E n)$ is the full inverse image of $G_E(M_E)$, i.e.,
\begin{eqnarray*}
G_E(M_E n) &=& \left\{ \tilde{g} \in \mbox{GL}_2(\Z /  M_E n \Z) : \tilde{g} \equiv g \mod M_E, \;\;
\mbox{for some $g \in G_E(M_E)$} \right\} \\
&=& \left\{ \tilde{g} = (\tilde{g_1}, \tilde{g_2}) \in \mbox{GL}_2(\Z /  n^2 \Z)
\times  \mbox{GL}_2(\Z /  (M_E/n) \Z) : \right. \\
&& \left. \tilde{g_1} \equiv g {\mod {n}}, \;
\tilde{g_2} \equiv g \mod {M_E/n}
\;\;
\mbox{for some $g \in G_E(M_E)$} \right\} ,
\end{eqnarray*}
where the second line follows from the Chinese Remainder Theorem and the fact that, in this case, $(n^2, (M_E/n))=1$ and $\tilde{g}$ is the usual unique lift of $(\tilde{g_1}, \tilde{g_2})$ to
$\mbox{GL}_2(\Z /  M_E n \Z)$.
\begin{comment} Notice that I do not say
that $\tilde{g_1} \equiv g_1 {\mod {n}},
\tilde{g_2} \equiv g_2 \mod {M_E/n}$, for some pair $(g_1, g_2) \in G_E(n) \times G_E(M_E/n)$
as this would be false, they have to come from the same $g \in G_E(M_E n)$.
\end{comment}
Since $G_E(n^2)$  is the
projection of $G_E(M_E n)$  into $\mbox{GL}_2(\Z /  n^2 \Z)$, we obtain
\begin{eqnarray} \nonumber
G_E(n^2)
&=& \left\{ \tilde{g_1} \in \mbox{GL}_2(\Z /  n^2 \Z) :
 \tilde{g_1} \equiv g {\mod {n}}
\;\;\mbox{for some $g \in G_E(M_E)$} \right\} \\ \label{GEn2}
&=& \left\{ \tilde{g_1} \in \mbox{GL}_2(\Z /  n^2 \Z) :
 \tilde{g_1} \equiv g {\mod {n}}
\;\;\mbox{for some $g \in G_E(n)$} \right\},
\end{eqnarray}
where the second line follows from our assumptions that $n \mid M_{E}$ and $G_E(n)$ is the projection of $G_E(M_E)$ modulo $n$.
From here, we may conclude that $G_E(n^2)$ is the full inverse image of $G_E(n)$. By (\ref{largeimage}), since $n \mid M_E$ and $n \neq M_E$, we have
 $G_E(n) = \mbox{GL}_2(\Z /  n \Z)$. Therefore,  by \eqref{GEn2}, we have
 \begin{eqnarray*}
 G_E(n^2)& =&  \left\{ \tilde{g_1} \in \mbox{GL}_2(\Z /  n^2 \Z) :
 \tilde{g_1} \equiv g {\mod {n}}
\;\;\mbox{for some $g \in \mbox{GL}_2(\Z /  n \Z)$} \right\}\\
& = &  \mbox{GL}_2(\Z /  n^2 \Z).
 \end{eqnarray*}

If $n$ is an odd squarefree positive integer, then by (\ref{E12Jones}),
we have $(n, M_E/n)=1$, which implies
that $G_E(n^2) = \mbox{GL}_2(\Z /  n^2 \Z)$.
Suppose that the squarefree integer $n$ is even. Then,  $n = 2m$ and $m$ is odd.
If $\nu_2(M_E)=1$, then $(n, M_E/n)=1$, and $G_E(n^2) = \mbox{GL}_2(\Z /  n^2 \Z)$.
If $\nu_2(M_E)=2$, then $(2n, M_E/2n)=1$. If $2n \neq M_E$, we have that
$G_E((2n)^2) = \mbox{GL}_2(\Z /  (2n)^2 \Z)$
which, by projection into $\mbox{GL}_2(\Z /  n^2 \Z)$, implies that
$G_E(n^2) = \mbox{GL}_2(\Z /  n^2 \Z)$.
Similarly, if $\nu_2(M_E)=3$, then $(4n, M_E/4n)=1$. If $4n \neq M_E$, we have that
$G_E((4n)^2) = \mbox{GL}_2(\Z /  (4n)^2 \Z),$
which implies that
$G_E(n^2) = \mbox{GL}_2(\Z /  n^2 \Z)$.

Therefore the only cases where  $G_E(n^2)$ may not equal $\mbox{GL}_2(\Z /  n^2 \Z)$ are those listed in the statement of our lemma.
\end{proof}
%******************************************

%*****CD: I changed the statement and the proof the Lemma below as we now have one more case, namely
%$n=M_E/4$ according to the new proof of lemma 6.3; this also addressed the parity issue
%raised by the referee: if we talk about $M_E/4$, we should make sure that $4 \mid M_E$...******

\begin{lemma} \label{boundratio}
Let $f(x,y)$ be any squarefree non-constant polynomial in $\Z[x,y]$, and let $E$ be a Serre curve.
Let $n$ be a squarefree integer in  $ \left\{ M_E, M_E/2, M_E/4 \right\} \cap \Z$.
Then for any $\varepsilon > 0$, we have
\begin{eqnarray} \label{upperbounds1}
\frac{|C_{E,f}(n^2)|}{|G_E(n^2)|} &\ll&
\frac{|C_{f}(n^2)|}{|\mathrm{GL}_2(\Z / n^2 \Z)|} \ll \frac{1}{M_E^{2 - \varepsilon}} .
\end{eqnarray}
%%\label{upperbounds2}
%%\frac{|C_{E,f}(M_E^2/4)|}{|G_E(M_E^2/4)|} &\ll&
%%\frac{|C_{f}(M_E^2/4)|}{|\mathrm{GL}_2(\Z / M_E^2/4 \Z)|} \ll \frac{1}{M_E^{2 - \varepsilon}}.
%%\end{eqnarray}
\end{lemma}
\begin{proof}
The first inequality of \eqref{upperbounds1} follows immediately since
$E$ is a Serre curve, and therefore
$|G_E(n)| \geq |\gl(\Z / n \Z)|/2$ for any $n$.
The
second inequality follows from Lemma \ref{upperbounds-n1} as $M_E$ is not divisible by the square of any odd prime.
\end{proof}

\textbf{Proof of Proposition \ref{average-serre-curves}}.
For $E \in \mathcal{S}(A,B)$,
we have
\begin{eqnarray} \label{difference}
C_{E,f}^{SF} - C_{f}^{SF}  &=& \sum_{\mathrm{GL}_2(\Z/n^2 \Z) \neq G_E(n^2)}  \mu(n) \left( \frac{|C_{E,f}(n^2)|}{|G_E(n^2)|} -
\frac{|C_{f}(n^2)|}{|\gl(\Z / n^2 \Z)|} \right).
\end{eqnarray}
We would like to detect  the squarefree integers $n$ such that $\gl(\Z/n^2\Z) \neq G_E(n^2)$.
If $(n, M_E)=1$ then by  \eqref{serre1}, $n$ is not counted in the sum.
Therefore we only need to consider those values of $n$ where $(n, M_E) \neq 1$, in which case we may write $n = n_1 n_2$ with $(n_1, M_E)=1$ and $n_2 \mid M_E$.
(Recall that $n$ is squarefree.)
Using the property given in \eqref{serre2}, we obtain
$$
G_E(n^2) = \gl(\Z / n_1^2 \Z) \times G_E(n_2^2),
$$
and
$$
\frac{|C_{E,f}(n^2)|}{|G_E(n^2)|} = \frac{|C_{f}(n_1^2)|}{|\gl(\Z / n_1^2 \Z)|}  \frac{|C_{E,f}(n_2^2)|}{|G_E(n_2^2)|}.
$$
Lemma \ref{whendifferent} gives us a set of conditions for the values of $M_E$ and $\Delta_{SF}$ that $|G_E(n_2^2)| \neq \mathrm{GL}_2(\Z/n^2 \Z)$
 can occur for squarefree values of $n$ when $E$ is a Serre curve defined over $\Q$. We will now describe how to bound $C_{E,f}^{SF} - C_{f}^{SF}$ in each of these instances.

In the case where $M_E = 2 | \Delta_{SF} |$ with $\Delta_{SF} \equiv 1 \mod 4$, we can use Lemma \ref{whendifferent} together with \eqref{difference} to show that
\begin{eqnarray} \label{case1}
C_{E,f}^{SF} - C_{f}^{SF}  &\ll& \sum_{\substack{\mu(n) \neq 0 \\ n = M_E n_1}}  \frac{|C_{E,f}(M_E^2)|}{|G_E(M_E^2)|}
\frac{|C_{f}(n_1^2)|}{|\gl(\Z / n_1^2 \Z)|}
+
\frac{|C_{f}(n^2)|}{|\gl(\Z / n^2 \Z)|}.
\end{eqnarray}
Similarly, when
$M_E = 4 | \Delta_{SF} |$ with $\Delta_{SF}$ odd, we have
\begin{eqnarray} \label{case2}
C_{E,f}^{SF} - C_{f}^{SF}  &\ll& \sum_{\substack{ \mu(n) \neq 0 \\ n = (M_E/2) n_1}}  \frac{|C_{E,f}(M_E^2/4)|}{|G_E(M_E^2/4)|} \frac{|C_{f}(n_1^2)|}{|\gl(\Z / n_1^2 \Z)|}  +
\frac{|C_{f}(n^2)|}{|\gl(\Z / n^2 \Z)|}
\end{eqnarray}
and  when
$M_E = 4 | \Delta_{SF} |$ with $\Delta_{SF}$ even, we have
\begin{eqnarray} \label{case3}
C_{E,f}^{SF} - C_{f}^{SF}  &\ll& \sum_{\substack{ \mu(n) \neq 0 \\ n = (M_E/4) n_1}}  \frac{|C_{E,f}(M_E^2/16)|}{|G_E(M_E^2/16)|} \frac{|C_{f}(n_1^2)|}{|\gl(\Z / n_1^2 \Z)|}  +
\frac{|C_{f}(n^2)|}{|\gl(\Z / n^2 \Z)|}.
\end{eqnarray}
In all other cases, we have
$$
C_{E,f}^{SF} - C_{f}^{SF} = 0.
$$

Using Lemma \ref{boundratio} in \eqref{case1},  \eqref{case2} and \eqref{case3}, we obtain
\begin{eqnarray*}
C_{E,f}^{SF} - C_{f}^{SF}  &\ll&  \frac{1}{M_E^{2-\varepsilon}} \sum_{n_1} \frac{1}{n_1^{2-\varepsilon}}
\ll \frac{1}{M_E^{2-\varepsilon}}.
\end{eqnarray*}

%Suppose that $n^2 = M_{E} m$, with $n | M_{E}$. Therefore $n\leq M_{E}$. Let me consider two different cases.

%CASE 1.  First assume that
%$n = M_{E}$, then in our summation we get $\mu(M_{E}$) (that will only contribute if $M_{E}$ is squarefree;i.e  $M_{E} = 2 \Delta$).

%CASE 2. Now assume that $n < M_{E}$. Since $n^2 = M_{E} m$ and $n | M_{E}$, we can write  $n = n_{1} n_{2}$ so that
%$M_{E} = n n_{1} = n_{2} n_{1}^2$. Recall  that $M_{E}$ is either twice or $4$ times a squarefree number (mind that factor $n_{1}^2$ here!). Therefore either $n_{1} = 1$ (which is CASE 1) or  $n_{1} = 2$. in CASE 2 we must have $n_{1} = 2$. Then $n = M_{E} / 2$.
%in this case $M_{E}$ is obviously four times a SF. This means CASE 1 and CASE 2 will not happen at the same time.

In order to complete our argument, we will need the following result from \cite{jones}: for any positive integer $k$,
\begin{equation}\label{E23Jones}
\frac{1}{4 AB} \sum_{\substack{|a| \leq A\\ |b|\leq B\\ 4a^3+27b^2 \neq 0}}
\frac{1}{|(4a^3 + 27b^2)_{SF}|^k} \ll \frac{1}{A} + \left( \frac{ \log{B} (\log{A})^7}
{B} \right)^{k(k+1)/2}.\end{equation}
From here, we may conclude that
\begin{eqnarray*}
\frac{1}{|\mathcal{C}(A , B)|} \sum_{E \in \mathcal{S}(A , B)} C_{E,f}^{ SF}
& = & \frac{|\mathcal{S}(A,B)|}{|\mathcal{C}(A , B)|} C_{f}^{SF} +
O \left( \frac{1}{A} + \left( \frac{ \log{B} (\log{A})^7}
{B} \right)^{3-\varepsilon} \right)\\
&\sim& C_{f}^{SF},
\end{eqnarray*}
since almost all elliptic curves are Serre curves (see  \cite{jonesthesis}); i.e., as $A,B \rightarrow \infty,$
\begin{equation*}
 \frac{|\mathcal{S}(A,B)|}{|\mathcal{C}(A , B)|} \sim 1.
\end{equation*}
This completes our proof of Proposition \ref{average-serre-curves}. \qed

Theorem \ref{thmaveragec} then follows from Proposition  \ref{lemma-nonserre} and Proposition
\ref{average-serre-curves}.

\textit{Acknowledgements.} This paper came out of work that began at the \textit{Women In Numbers 2} workshop. We would like to thank the \textit{WIN 2} organizers and the Banff International Research Station for providing us with the opportunity to collaborate. We would also like to express our gratitude to Min Lee, who participated in the early stages of this research; her notes were very
helpful in the preparation of this manuscript. Finally, we would like to thank Nathan Jones and the anonymous referee for their careful reading of the paper and for providing helpful comments.

\providecommand{\bysame}{\leavevmode\hbox
to3em{\hrulefill}\thinspace}
\providecommand{\MR}{\relax\ifhmode\unskip\space\fi MR }
\providecommand{\nMRhref}[2]{%
  \href{http://www.ams.org/mathscinet-getitem?mr=#1}{#2}
} \providecommand{\href}[2]{#2}


\begin{thebibliography}{10}

\bibitem{Baier}
S. Baier, \emph{The Lang-Trotter conjecture on average}, J. Ramanujan Math. Soc. \textbf{22} (2007), 299-314.

\bibitem{BaCoDa} A. Balog, A.C. Cojocaru and C. David, \textit{Average twin prime conjecture for elliptic curves}, Amer. J. Math. \textbf{133} no. 5 (2011), 1179-1229.

\bibitem{BaSh} B. Banks and I. Shparlinski, \textit{Sato-Tate, cyclicity, and divisibility statistics for elliptic curves of small height}, Israel J.
Math., to appear.

\bibitem{BBIJ05} J. Battista, J. Bayless, D. Ivanov, and K. James, \textit{Average Frobenius distributions for elliptic curves with nontrivial rational torsion}, Acta Arith. \textbf{119} no. 1 (2005), 81-91.

\bibitem{CFJ+11} N. Calkin, B. Faulkner, K. James, M. King, and D. Penniston, \textit{Average Frobenius distributions for elliptic curves over abelian extensions}, Acta Arith. \textbf{149} no. 3 (2011), 215-244.


\bibitem{cojocaru-thesis} A.C. Cojocaru, \textit{Cyclicity of elliptic curves modulo p}, Ph.D. thesis, Queen's University (2002).

\bibitem{cojocaru2} A.C. Cojocaru, \textit{Questions about the reductions modulo primes of an elliptic curve}, CRM Proceedings and Lecture Notes (2004).

%%\bibitem{cojocaru-murty} A.C. Cojocaru and R. Murty, \textit{Cyclicity of elliptic curves modulo $p$ and elliptic curve analogues of Linnik's problem}, Math. Ann. \textbf{330} (2004), 601-625.

\bibitem{cojocaru}
A.C. Cojocaru, \textit{Squarefree orders for CM elliptic curves modulo p}, Math. Ann. \textbf{342} no. 3 (2008), 587-615.

\bibitem{CoIwJo}
A.C. Cojocaru, H. Iwaniec and N. Jones, \emph{The average asymptotic behaviour of the Frobenius fields of an elliptic curve}, preprint.

\bibitem{DP99}
C. David and F. Pappalardi, \emph{Average Frobenius distributions of elliptic curves}, Int. Math. Res. Notices \textbf{4} (1999), 165-183.

\bibitem{DP04} C. David and F. Pappalardi, \textit{Average Frobenius distribution for inerts in Q(i)}, J. Ramanujan Math. Soc. \textbf{19} no. 3 (2004), 181-201.

\bibitem{du}
C. David and J. Jim\'enez Urroz, \textit{Squarefree discriminants of Frobenius rings}, Int. J. Number Theory \textbf{6} no. 5 (2010), 1391-1412.

%%\bibitem{ds}
%%C. David and E. Smith, \textit{Elliptic curves with a given number of points over finite fields}, to appear,
%%Compositio Mathematica.

%%\bibitem{ds2}
%%C. David and E. Smith, \textit{A Cohen-Lenstra Phenomenon for Elliptic curves}, preprint, %%arXiv:1206.1585.

\bibitem{elkies}
N. Elkies, \emph{The existence of infinitely many supersingular primes for every elliptic curve over $\Q$}, Invent. Math. \textbf{89} (1987), 561-568

\bibitem{FM}
E. Fouvry and R. Murty, \emph{On the distribution of supersingular primes}, Canadian J. Math. \textbf{48} no. 1 (1996), 81-104.

\bibitem{gekeler}
E.-U. Gekeler, \textit{Statistics about elliptic curves over finite prime fields}, Manuscripta Math. \textbf{127} (2008) no. 1, 55-67.

\bibitem{GH}
P. Griffifths and J. Harris, \textit{Principles of Algebraic Geometry}, Wiley \& Sons (1978).


\bibitem{Jam04} K. James, \textit{Average Frobenius distributions for elliptic curves with $3$-torsion}. J. Number Theory \textbf{109} no. 2 (2004), 278-298.

\bibitem{JS11} K. James and E. Smith, \textit{Average Frobenius distribution for elliptic curves defined over Finite Galois extensions of the rationals}, Math. Proc. Cambridge Philos. Soc. \textbf{150} no. 3 (2011) 439-458.

\bibitem{jones}
N. Jones, \textit{Averages of elliptic curve constants}, Math. Ann. \textbf{345} (2009) no. 3, 685-710.

\bibitem{jonesthesis}
N. Jones, \textit{Almost all elliptic curves are Serre curves}, Trans. Amer. Math. Soc.  \textbf{362} (2010), 1547-1570.

\bibitem{koblitz}
N. Koblitz, \emph{Primality of the number of points on an elliptic curve over a finite field}, Pacific J. Math. \textbf{131} no. 1 (1988), 157-165.

\bibitem{lt}
S. Lang and H. Trotter, \textit{Frobenius distributions in $GL_2$-extensions}, Lecture Notes in Mathematics, vol. 504, Springer-Verlag, Berlin, 1976.

\bibitem{lo}
J. Largarias and A. Odlyzko, \emph{Effective version of the Chebotararev Density Theorem}, Algebraic Number Fields (A. Fr\"{o}hlich edit.), NY, Academic Press (1977), 409-464.

\bibitem{Serre} J.-P. Serre, \textit{Proprietes galoisiennes des points d'ordre fini des courbes elliptiques}, Invent. Math. \textbf{15} (1972), 259-331.

\bibitem{Serre1}
J.-P. Serre, \emph{Quelques applications du th\'eor\'eme de densit\'e de Chebotarev}, Inst. Hautes \'{E}tudes Sci. Publ. Math. \textbf{54} (1981), 323-401.

\bibitem{stark} H. M. Stark, \textit{Some effective cases of the Brauner-Siegel theorem}, Invent. Math. \textbf{23} (1974), 135-152.

\bibitem{zywina}
D. Zywina, \textit{Bounds for Serre's open image theorem}, preprint.

\end{thebibliography}
\end{document}